\documentclass[11pt, a4paper]{amsart}

\usepackage{amsmath}
\usepackage{amssymb}
\usepackage{graphicx}
\usepackage{color}
\usepackage{url}
\usepackage{blkarray}
\usepackage{hyperref}
\hypersetup{
    colorlinks,
    citecolor=black,
    filecolor=black,
    linkcolor=black,
    urlcolor=black
}

\newtheorem{thm}{Theorem}
\newtheorem*{que*}{Question}
\newtheorem{prop}[thm]{Proposition}
\newtheorem{lem}[thm]{Lemma}

\newtheorem{remark}[thm]{Remark}

\newtheorem{cor}[thm]{Corollary}

\begin{document}

\title{Divide knots of maximal genus defect}

\author{Livio Liechti}
\address{Department of Mathematics\\
University of Fribourg\\
Chemin du Mus\'ee 23\\
1700 Fribourg\\Switzerland}
\email{livio.liechti@unifr.ch}

\begin{abstract} 
We construct divide knots with arbitrary smooth four-genus but topological 
four-genus equal to one. In particular, for strongly quasipositive fibred knots, the 
ratio between the topological and the smooth four-genus can be arbitrarily close to zero. 
\end{abstract}

\maketitle

\section{Introduction}

The smooth four-genus of algebraic knots equals the usual Seifert genus by 
Kronheimer and Mrowka's resolution of the Thom conjecture~\cite{KM}. This 
is far from being true for the topological four-genus. For instance, if the two 
parameters~$p$ and~$q$ tend to infinity, the ratio between the topological and the 
smooth four-genus of the torus knot~$T_{p,q}$ is expected to approach~$\frac{1}{2}$. 
At the moment, the best upper bound for this limit is~$\frac{14}{27}$, provided by 
Baader, Banfield and Lewark~\cite{BBL}. On the 
other hand, it is well-known that the limit cannot lie below~$\frac{1}{2}$.
This follows from Gordon, Litherland and Murasugi's signature formulas for torus 
knots~\cite{GLM} and the fact that the signature invariant is a lower bound for twice 
the topological four-genus by a result of Kauffman and Taylor~\cite{KT}. 
By Shinohara's cabling relation~\cite{Sh}, the lower bound~$\frac{1}{2}$ for the 
ratio between the topological and the smooth four-genus extends to the class 
of all algebraic knots. For the more general class of positive braid knots, this 
ratio is bounded from below by~$\frac{1}{8}$ due to a result of Feller~\cite{Feller}, 
and even for positive knots, it is bounded from below by~$\frac{1}{12}$ due to a 
result of Baader, Dehornoy and the author~\cite{BDL}.

In this article, we consider another natural generalisation of algebraic knots: divide knots, 
introduced by A'Campo~\cite{AC4, AC3}. This class of knots lies within the class of strongly 
quasipositive fibred knots, but not every divide knot is positive. 
Our aim is to show that for divide knots, and hence also for strongly quasipositive fibred knots, 
the ratio between the topological and the smooth four-genus 
can be arbitrarily close to zero. More precisely, we prove the following result.

\begin{thm} 
\label{trefoil_thm}
For every positive integer~$g$, there exists a divide knot with smooth four-genus equal 
to~$g$ and with topological four-genus equal to one. 
\end{thm}

This result is optimal in the sense that the topological four-genus of every nontrivial divide knot is 
at least one. This follows from the fact that the signature of a nontrivial divide knot is bounded from 
below by two. 

Our proof of Theorem~\ref{trefoil_thm} consists of the following steps. 
First of all, the smooth four-genus of a divide knot equals the 
usual Seifert genus. This is the content of Rudolph's extension of the Thom conjecture to strongly 
quasipositive knots~\cite{Rudolph}. Another proof for divide knots is provided by A'Campo~\cite{AC4}. 
Hence, the only point is to show that the topological four-genus equals one. 
We do this for the following explicit examples of divide knots of growing Seifert genus: 
let~$K_n$ be the divide knot obtained by the snail divide with~$n$ double points, 
see Figure~\ref{snaildivides} for the example~$n=1,2,3,4$. We refer to Section~\ref{basics} for the 
definition of divide knots. The first two knots of the sequence are the trefoil knot~$3_1$ and the 
knot~$10_{145}$. 
\begin{figure}[h]
\begin{center}
\def\svgwidth{300pt}
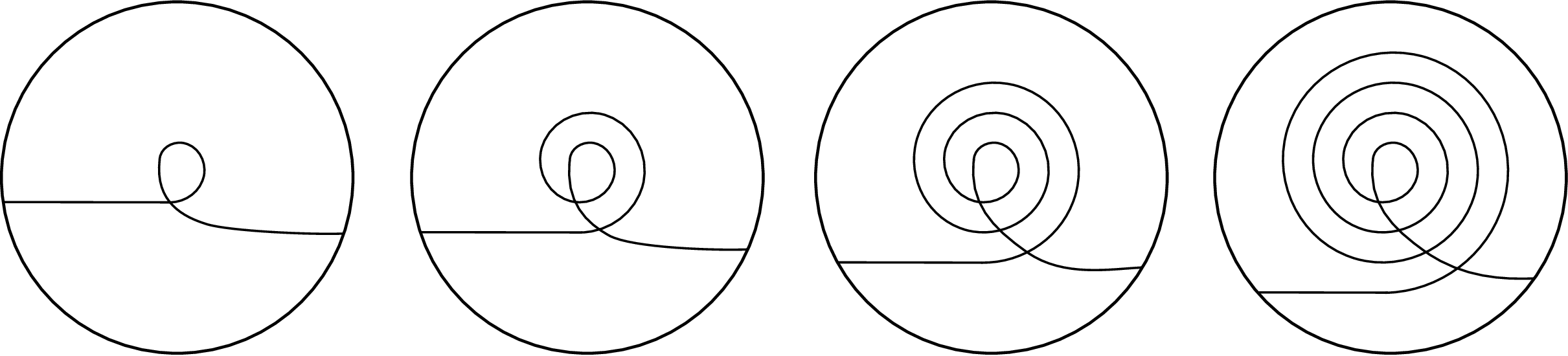
\caption{The snail divides with one, two, three and four double points, from left to right.}
\label{snaildivides}
\end{center}
\end{figure}
We analyse the Seifert form of the canonical genus-minimising 
Seifert surface of~$K_n$, and use an algebraic criterion due to Baader, Feller, Lewark and the 
author~\cite{BFLL} to bound the topological four-genus from above. Unsurprisingly, this criterion is 
based on Freedman's disc theorem~\cite{Free}. 

We state the consequence we get for the minimal ratio between the smooth and the topological 
four-genus separately as follows.

\begin{cor}
The ratio between the topological and the smooth four-genus can be arbitrarily close to zero  
in any of the following classes of knots: divide knots, positive trefoil plumbings, strongly quasipositive 
fibred knots.
\end{cor}

The only nontrivial extension with respect to Theorem~\ref{trefoil_thm} is the statement for 
positive trefoil plumbings. However, this follows directly from the fact that the examples~$K_n$ are
positive trefoil plumbings due to a result of the author~\cite{L1}. Unfortunately, this result features only 
in the first arxiv version of the article, and is not contained in the published version.

\begin{remark}\emph{
While the knots~$K_n$ might as well be the only divide knots with topological four-genus equal to one, 
there are probably infinitely many strongly quasipositive fibred knots of this kind, for any large enough 
Seifert genus: realising~$K_n$ abstractly as a positive trefoil plumbing, it should be possible to plumb 
the last trefoil in infinitely many ways but with the same Seifert form, similarly to Misev's construction of 
infinitely many strongly quasipositive fibred knots with the same Seifert form~\cite{Misev}. 
}\end{remark}

If one drops the assumption of fibredness, 
a much stronger result is known: 
Borodzik and Feller recently showed that every 
knot is topologically concordant to a strongly quasipositive one~\cite{BF}.
However, such a result cannot be expected to hold restricted to strongly quasipositive fibred knots, 
as it distinctly does not do so in the setting of algebraic concordance. A partial result in this direction is 
due to Yozgyur~\cite{Yozgyur}: there exist knots that are not topologically 
concordant to any L-space knot (which form a subclass of strongly quasipositive fibred knots).
\smallskip

\noindent
\textbf{Organisation}. We provide basic definitions and properties of divide knots in Section~\ref{basics}, 
and we prove Theorem~\ref{trefoil_thm} in Section~\ref{proof}. 
\smallskip

\noindent
\textbf{Acknowledgements}. I thank Sebastian Baader for asking about the topological four-genus of divide knots, 
and I thank Peter Feller for helpful comments and references. I also thank the anonymous referees for their close 
reading and their constructive comments.

\section{The Seifert form of a divide knot}
\label{basics}

\subsection{Divide knots}
We briefly recall the definition of divide knots. For more details, we refer to A'Campo's original articles~\cite{AC4, AC3} 
or Baader and Dehornoy~\cite{BD}.
Let~$D$ be the closed unit disc and let~$P$ be the image of a relative smooth arc immersed generically in~$D$. 
We identify the tangent bundle~$T(D)$ to~$D$ with the product~$D\times\mathbf{R}^2$ and consider its 
unit sphere \[ST(D)=\{(x,v) \in T(D): ||x||^2+||v||^2 = 1\} \cong \mathbf{S}^3.\]
The divide knot~$K(P)$ is defined to be the set of vectors in~$ST(D)$ based at and tangent to~$P$.

Divide knots are fibred by a result of A'Campo~\cite{AC4}. In particular, they have a canonical 
genus-minimising Seifert surface~$\Sigma$. It is obtained as follows. Colour the complement~$D\setminus P$ 
into black and white regions in checkerboard fashion. The surface~$\Sigma$ then consists of the following 
points. For every ordinary point~$p$ of~$P$, it contains all the vectors of~$ST(D)$ that are based at~$p$ and do not 
point towards a white region. For every double point~$p$, it contains all the vectors of~$ST(D)$ that are based at~$p$. 
Figure~\ref{snail2fibre} is inspired by Figure~1.3 by Baader and Dehornoy~\cite{BD}, and shows the canonical 
Seifert surface for the snail divide with two double points. 

\begin{figure}[h]
\begin{center}
\def\svgwidth{270pt}
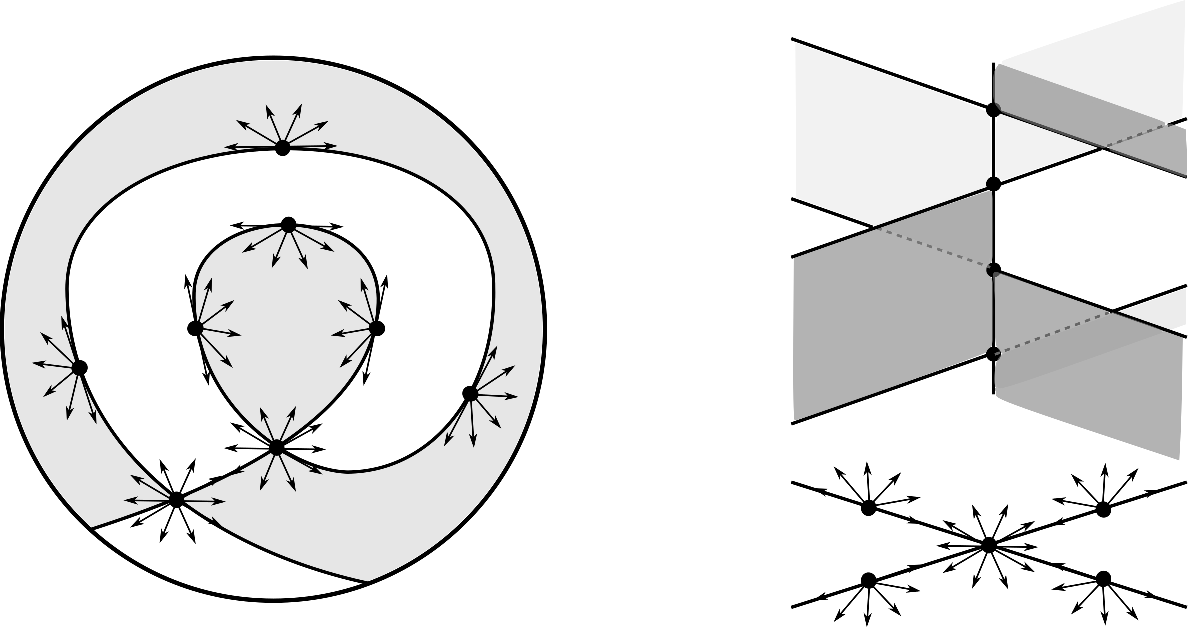
\caption{On the left, an illustration of the canonical Seifert surface~$\Sigma$ of the snail 
divide with two double points. On the right, a portion of the surface is depicted around a double point, 
where the height is given by the argument of the corresponding vector, measured from~$-\pi$ to~$\pi$. 
The top and the bottom of the surface have to be identified.}
\label{snail2fibre}
\end{center}
\end{figure}

Ishikawa realised the canonical Seifert surface as a plumbing of positive Hopf bands~\cite{Ishikawa}. More precisely, 
the core curve of the Hopf bands in the plumbing construction are the following: There is one basis curve for 
each double point~$p$, consisting of orienting its lift to~$ST(D)$. We choose to orient the vectors clockwise, 
so that the curve corresponding to the double point in Figure~\ref{snail2fibre} is oriented downwards. 
There is one basis curve for each inner region of~$D\setminus P$, consisting of a curve in~$ST(D)$ 
that projects to a curve running clockwise around the boundary of the region. If the region is black, 
then we choose a lift that always points into the region; if the region is white, we choose a lift that always 
points outside of the region. Figure~\ref{Ialg} depicts portions of these basis curves around a double point of the divide.

\begin{figure}[h]
\begin{center}
\def\svgwidth{300pt}
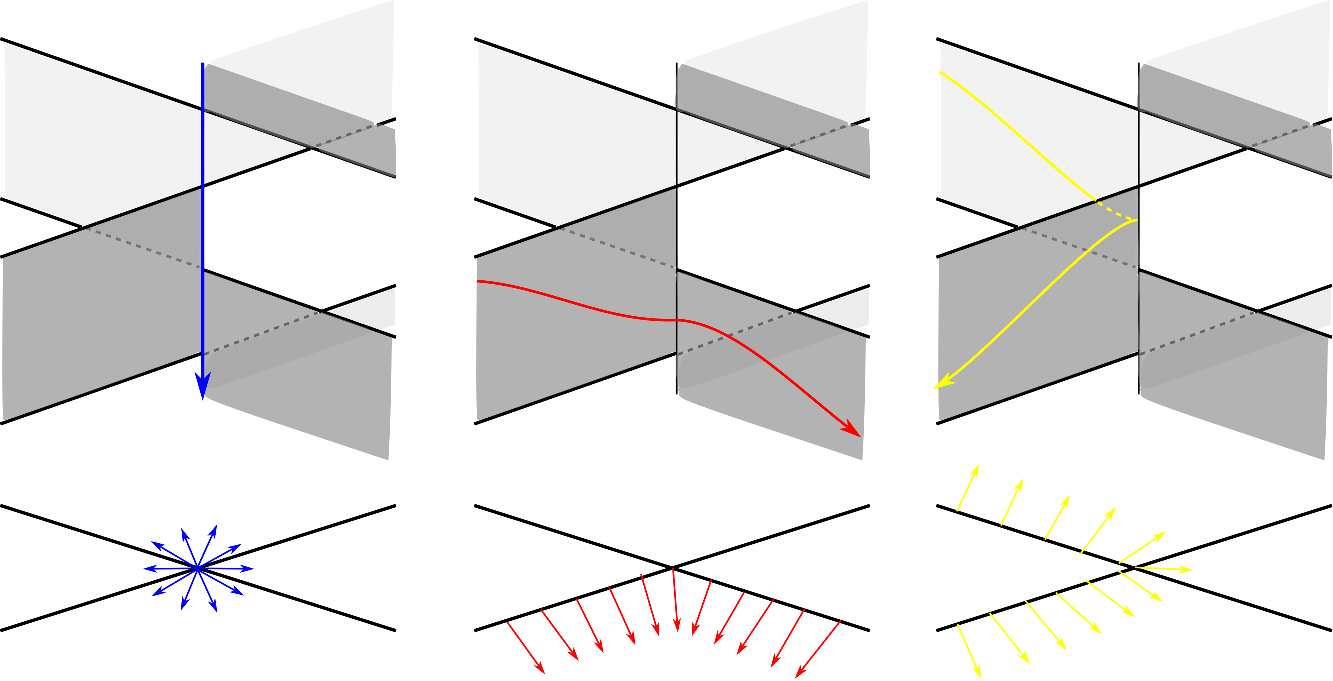
\caption{Portions of Ishikawa's basis curves: on the left corresponding to a double point (blue),
in the middle corresponding to a black region (red), and on the right corresponding to a white region (yellow).}
\label{Ialg}
\end{center}
\end{figure}

\subsection{The Seifert form}
For a Seifert surface~$\Sigma$ of a knot~$K$, the associated Seifert form~$S$ is defined as follows. 
For two simple closed curves~$\alpha$ and~$\beta$ in~$\Sigma$, define~$S(\alpha,\beta)$ to be the 
linking number of~$\alpha$ with~$\beta$, where~$\beta$ is  
slightly pushed off~$\Sigma$ in the positive normal direction. 
Choosing a basis of~$\mathrm{H}_1(\Sigma;\mathbf{Z})$ that consists entirely of simple closed curves, 
this definition extends to a bilinear form on~$\mathrm{H}_1(\Sigma;\mathbf{Z})$, the Seifert form~$S$.

Now, let~$K$ be a divide knot with canonical Seifert surface~$\Sigma$. We describe the Seifert 
form~$S$ by the values it takes on the basis of~$\mathrm{H}_1(\Sigma;\mathbf{Z})$ 
obtained by Ishikawa's plumbing construction~\cite{Ishikawa}. In order to do so, we need to choose 
an orientation of~$\Sigma$. Our convention is that the projection to~$D$ of the positive normal 
vector to~$\Sigma$ points into the adjacent black region. With this convention, the darker side 
of the surface depicted in Figure~\ref{Ialg} is the positive one.

\begin{lem}
\label{Divideform}
Let~$\alpha$ be a curve corresponding to a black region~$B$ of the divide, let~$\beta$ be 
a curve corresponding to a white region~$W$ of the divide, and let~$\gamma$ be a curve 
corresponding to a double point~$v$ of the divide. 
Assume further that~$B$ and~$W$ have~$k$ common edges, that~$v$ appears~$n$ and~$m$
times, respectively, when tracing the boundary of~$B$ and~$W$, respectively.
Then we have
\begin{enumerate}
\item[(i)] $S(\alpha,\alpha) = S(\beta,\beta) = S(\gamma,\gamma) = 1,$
\item[(ii)] $S(\beta, \alpha) = S(\gamma, \alpha) = S(\beta, \gamma) = 0,$
\item[(iii)] $S(\alpha, \beta) = k,$
\item[(iv)] $S(\alpha,\gamma) = n,$
\item[(v)] $S(\gamma, \beta) = m.$
\end{enumerate}
\end{lem}

We note that while the values of~$n$ and~$m$ are either~$0$,~$1$ or~$2$ by construction, 
the value of~$k$ can be arbitrarily large.

\begin{proof}
%First of all, we note that these values give the correct symmetrised Seifert form, as described 
%by Baader and Dehornoy~\cite{BD}. In particular, if we can show that 
%\begin{enumerate}
%\item $S(\alpha,\alpha) = S(\beta,\beta) = S(\gamma,\gamma) = 1$,
%\item $S(\beta, \alpha) = S(\gamma, \alpha) = S(\beta, \gamma) = 0$,
%\end{enumerate}
%then we are done. 
The statement~(i) follows directly from the fact that the curves~$\alpha,\beta$ 
and~$\gamma$ are the core curves of positive Hopf bands in Ishikawa's plumbing construction~\cite{Ishikawa}. 
The statement~(ii) follows from the following observation, given by Baader and Dehornoy~\cite{BD}:
if the projections of two loops in~$ST(D)$ to~$D$ are disjoint, then the loops have linking number zero. 
The reason for this is that every diameter in~$D$ lifts to a 2-sphere in~$ST(D)$, and
hence loops with disjoint projections are separated by a 2-sphere and have zero linking.
Now, if we push~$\alpha$ off~$\Sigma$ in the positive direction, that is, towards the black face, 
then its projection becomes disjoint with the projections of~$\beta$ and~$\gamma$. This explains 
the first two zeros in the statement~(ii). For the last zero, we note that pushing~$\gamma$ 
off~$\Sigma$ in the positive direction yields the same linking number as pushing~$\beta$ 
off~$\Sigma$ in the negative direction. But if we push~$\beta$ in the negative direction, 
that is, towards the white face, then again the projections of~$\beta$ and~$\gamma$ to~$D$ 
become disjoint. 

We now argue how to obtain~(iii),(iv) and~(v) from~(i) and~(ii).
For any pair of transverse, oriented, simple closed curves~$a$ and~$b$ in~$\Sigma$,
we have~$$i_\mathrm{alg}(a,b) = S(b,a) - S(a,b),$$ where the algebraic intersection number~$i_\mathrm{alg}$
is a signed count of the number of intersections: an intersection gets counted with a positive sign if the orientation of the intersection matches the orientation 
of the surface, and with a negative sign otherwise. Given the definition of the basis curves of type~$\alpha, \beta$ and~$\gamma$, it is a direct verification 
that the values of the algebraic intersection numbers are given by~$i_\mathrm{alg}(\alpha,\beta) = -k$,~$i_\mathrm{alg}(\alpha,\gamma) = -n$ 
and~$i_\mathrm{alg}(\gamma,\beta) = -m$, compare with Figure~\ref{Ialg} (and recall that with our 
convention, the darker side of the surface is the positive one).    
In particular, by the above formula for the algebraic intersection number,~(iii),(iv) and~(v) follow from the values computed in~(i) and~(ii). 
\end{proof}

\section{Proof of Theorem~\ref{trefoil_thm}}
\label{proof}

Let~$K_n$ be the divide knot obtained by the {snail divide} with~$n$ double points, as defined in Figure~\ref{snaildivides}. 
Knowing Ishikawa's basis for the first homology of the canonical genus-minimising Seifert surface~$\Sigma$, we directly 
see that~$K_n$ is of genus~$n$. The technical ingredient to the proof of Theorem~\ref{trefoil_thm}
is the following.

\begin{prop}
\label{Alextrivial}
Let~$\Sigma$ be the canonical Seifert surface for the knot~$K_n$.
Then, the first homology~$\mathrm{H}_1(\Sigma;\mathbf{Z})$ has a subgroup~$V$ of rank~$2n-2$ 
such that for a matrix~$A$ of the Seifert form of~$\Sigma$ 
restricted to~$V$,~$\det(tA-A^\top)$ is a unit in~$\mathbf{Z}[t^{\pm1}]$.
\end{prop}

Given Proposition~\ref{Alextrivial}, we finish the proof of Theorem~\ref{trefoil_thm} by applying 
the following proposition of Baader, Feller, Lewark and the author~\cite{BFLL} to the subgroup~$V$.

\begin{prop}[Proposition~3 in~\cite{BFLL}]
\label{bound}
Let~$L$ be a link with a Seifert surface~$\Sigma$ and associated Seifert form~$S$. 
If~$V\subset \mathrm{H}_1(\Sigma;\mathbf{Z})$ is a subgroup so that for a matrix~$A$ of~$S$ 
restricted to~$V$,~$\det(tA-A^\top)$ is a unit in~$\mathbf{Z}[t^{\pm1}]$, then the topological four-genus of~$L$ 
is bounded from above by~$g(\Sigma)-\mathrm{rk}(V)/2$.
\end{prop}

Applying Proposition~\ref{bound} to the subgroup~$V$ implies that the topological four-genus of~$K_n$ 
is bounded from above by~$n-\frac{2n-2}{2} = 1$. By a result of Kauffman and Taylor~\cite{KT}, the signature 
is a lower bound for twice the topological four-genus. Therefore, the equality~$g_4^\mathrm{top}(K_n) = 1$ 
follows from the fact that the signature of any nontrivial divide knot is bounded from below by~2. 
The argument for this is straightforward and sketched by the author~\cite{L1}. 
It remains to prove~Proposition~\ref{Alextrivial}.

\begin{proof}[Proof of Proposition~\ref{Alextrivial}]
In order to obtain the canonical Seifert surface for the snail divide knot~$K_n$, we must checkerboard 
colour the complement of the snail divide.
We do so in such a way that the innermost region is coloured black, as shown in Figure~\ref{snail4} for the case~$n=4$.
\begin{figure}[h]
\begin{center}
\def\svgwidth{130pt}
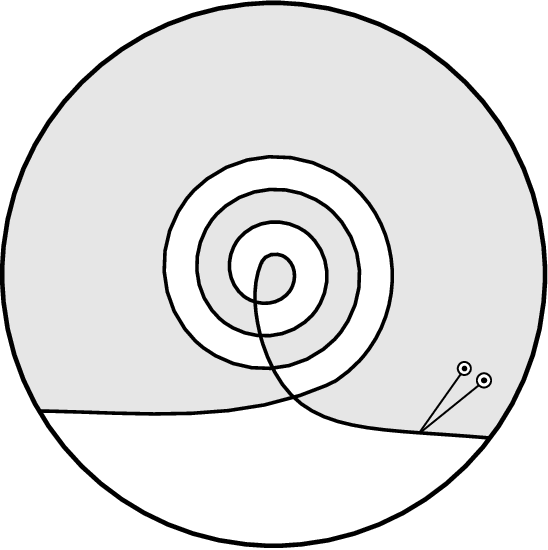
\caption{The snail divide with four double points and our chosen checkerboard colouring of the complement.}
\label{snail4}
\end{center}
\end{figure}
We now recall the basis of the canonical Seifert surface of a divide knot obtained by 
Ishikawa's plumbing construction. There is one basis curve for every inner region of~$D\setminus P$, 
and one basis curve for every double point. In our case of the snail divide knot~$K_n$, 
let~$\alpha_1,\dots,\alpha_n$ be the curves corresponding to the inner regions, from inside out. 
Furthermore, let~$\gamma_1,\dots,\gamma_n$ be the curves corresponding to the double points,
from inside out. Let~$V$ be the subgroup of~$\mathrm{H}_1(\Sigma;\mathbf{Z})$ generated by~$a_1,\dots,a_{n-1}, b_1,\dots,b_{n-1}$, 
where~$a_i = \alpha_{i+1} - \gamma_i$ and~$b_i = \gamma_{i+1}$. First of all, we note that these 
elements together with~$\alpha_1$ and~$\gamma_1$ form a basis of~$\mathrm{H}_1(\Sigma;\mathbf{Z})$,
so indeed~$V$ is a subgroup of rank~$2n-2$. We now describe enough of a matrix~$A$ for the Seifert form~$S$
restricted to~$V$ to show that~$\det(tA-A^\top)$ is a unit in~$\mathbf{Z}[t^{\pm 1}].$ \smallskip

\noindent
Using Lemma~\ref{Divideform}, we compute the following values of the Seifert form~$S$:

\begin{enumerate}
\item $S(a_i,a_j) = 0$ for all~$i,j$,
\item $S(a_i,b_i) = 0$ if~$i$ is odd and~$S(a_i,b_i) = 1$ if~$i$ is even,
\item $S(b_i,a_i) = 1$ if~$i$ is odd and~$S(b_i,a_i)=0$ if~$i$ is even,
\item $S(a_i,b_j) = 0$ and~$S(b_j,a_i) = 0$ for~$j>i$.
\end{enumerate}

\begin{figure}[h]
\begin{center}
\def\svgwidth{300pt}
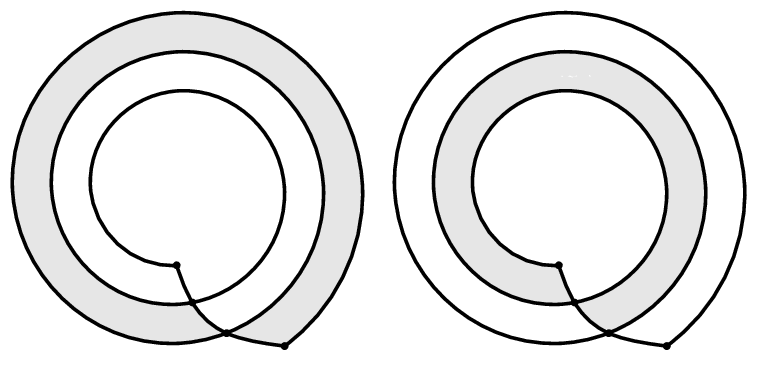
\caption{Two adjacent regions of a snail divide, with checkerboard colouring depending on whether~$i$ 
is even (left) or odd (right). The double points and the regions are labelled with the curves that correspond 
to them in Ishikawa's plumbing construction.}
\label{seifertcompute}
\end{center}
\end{figure}

\noindent
Even though the calculation of these values is straightforward, we provide it in full detail.
\smallskip

\noindent
\emph{Calculation of~(1):}
We first compute 
\begin{align*} S(a_i,a_i) &= S(\alpha_{i+1} - \gamma_i, \alpha_{i+1} - \gamma_i) \\
 &= S(\alpha_{i+1}, \alpha_{i+1}) - S(\alpha_{i+1},\gamma_i) - S(\gamma_i, \alpha_{i+1}) + S(\gamma_i,\gamma_i)\\
 & = 1 - S(\alpha_{i+1},\gamma_i) - S(\gamma_i, \alpha_{i+1}) + 1.
\end{align*}
Now, depending on whether~$\alpha_{i+1}$ corresponds to a black or a white region, respectively, 
we have~$S(\alpha_{i+1},\gamma_i) = 2$ and~$S(\gamma_i, \alpha_{i+1}) = 0$ or vice versa, respectively. 
This follows from Lemma~\ref{Divideform} and the fact that the vertex~$v$ corresponding to~$\gamma_i$ 
appears twice along the boundary of the region corresponding to~$\alpha_{i+1}$, see Figure~\ref{seifertcompute}. 
In any case, we obtain~$S(a_i,a_i)=0$. 
Next, we note that~$S(a_i,a_{j}) = S(a_{j},a_i) = 0$ for~$|i-j|\ge 2$. This follows directly from the definition of the~$a_i$ 
and Lemma~\ref{Divideform}. In order to finish the proof of~(1), the only values we have to compute 
are~$S(a_i,a_{i+1})$ and~$S(a_{i+1},a_i)$. We have
\begin{align*} S(a_i,a_{i+1}) &= S(\alpha_{i+1} - \gamma_i, \alpha_{i+2} - \gamma_{i+1}) \\
 &= S(\alpha_{i+1}, \alpha_{i+2}) - S(\alpha_{i+1},\gamma_{i+1}) - S(\gamma_i, \alpha_{i+2}) + S(\gamma_i,\gamma_{i+1})\\
 & = S(\alpha_{i+1}, \alpha_{i+2}) - S(\alpha_{i+1},\gamma_{i+1}) - S(\gamma_i, \alpha_{i+2}).
\end{align*}
If~$\alpha_{i+1}$ corresponds to a black region, then by Lemma~\ref{Divideform} this sum reads as follows:~$2 - 1 - 1 = 0$. 
On the other hand, if~$\alpha_{i+1}$ corresponds to a white region, then the sum reads~$0-0-0 = 0$. 
Similarly, we have 
\begin{align*} S(a_{i+1},a_i) &= S(\alpha_{i+2} - \gamma_{i+1}, \alpha_{i+1} - \gamma_{i}) \\
 &= S(\alpha_{i+2}, \alpha_{i+1}) - S(\alpha_{i+2},\gamma_{i}) - S(\gamma_{i+1}, \alpha_{i+1}) + S(\gamma_{i+1},\gamma_{i})\\
 & = S(\alpha_{i+2}, \alpha_{i+1}) - S(\alpha_{i+2},\gamma_{i}) - S(\gamma_{i+1}, \alpha_{i+1}).
\end{align*}
If~$\alpha_{i+2}$ corresponds to a black region, then by Lemma~\ref{Divideform} this sum reads as follows:~$2 - 1 - 1 = 0$. 
On the other hand, if~$\alpha_{i+2}$ corresponds to a white region, then the sum reads~$0-0-0 = 0$. 
This finishes the proof of~(1). \smallskip

\noindent
\emph{Calculation of~(2):} We have
\begin{align*} S(a_i,b_i) &= S(\alpha_{i+1} - \gamma_{i}, \gamma_{i+1}) \\
 &= S(\alpha_{i+1}, \gamma_{i+1}) - S(\gamma_{i}, \gamma_{i+1}) = S(\alpha_{i+1}, \gamma_{i+1}).
\end{align*}
By Lemma~\ref{Divideform}, this equals~~$0$ or $1$, respectively, if~$\alpha_{i+1}$ corresponds to a white or a black region, 
respectively, that is, when~$i$ is odd or even, respectively. This proves~(2).
\smallskip

\noindent
\emph{Calculation of~(3):} We have
\begin{align*} S(b_i,a_i) &= S(\gamma_{i+1}, \alpha_{i+1} - \gamma_{i}) \\
 &= S(\gamma_{i+1}, \alpha_{i+1}).
\end{align*}
By Lemma~\ref{Divideform}, this equals~~$1$ or $0$, respectively, if~$\alpha_{i+1}$ corresponds to a white or a black region, 
respectively, that is, when~$i$ is odd or even, respectively. This proves~(3).
\smallskip

\noindent
\emph{Calculation of~(4):} We note that if~$j>i$, then 
\begin{align*} S(a_i,b_j) &= S(\alpha_{i+1} - \gamma_{i}, \gamma_{j+1}) \\
 &= S(\alpha_{i+1}, \gamma_{j+1}) = 0,
\end{align*}
since no vertex corresponding to a curve~$\gamma_{j+1}$ appears in the boundary of a region 
corresponding to~$\alpha_{i+1}$ if~$j>i$, compare with Figure~\ref{seifertcompute}. We have
\begin{align*} S(b_j,a_i) &= S(\gamma_{j+1},\alpha_{i+1} - \gamma_{i}) \\
 &= S(\gamma_{j+1}, \alpha_{i+1}) = 0
\end{align*}
for the same reason. This proves~(4).
\smallskip

While the values computed above in~(1)-(4) do not give a complete description 
of the Seifert form~$S$ restricted to~$V$, they suffice to deduce that the matrix~$A$ for~$S$ 
restricted to~$V$ with respect to the basis~$a_1,\dots,a_{n-1},b_1,\dots,b_{n-1}$ is of the form
\[
A =\ \begin{blockarray}{crrrrrrrc}
    \begin{block}{(crrr|rrrrc)}
      \ & & \ & \  & 0 &  & & \\
      \ & & \ & \ & \ast  & 1 & & \\ 
       & \ &  & \ & \ast & \ast & 0 & \\
       & \ & & \ & \ast & \ast & \ast & \ddots \\  \BAhhline{-------------~}
      1 & \ast & \ast & \ast  &  &  & & \\
       & 0 & \ast & \ast &  &  & &  \\
      & & 1 & \ast&   &  & \ast & \\
       & &  & \ddots &   &  & & \\
    \end{block}
  \end{blockarray}
\]
where the blocks are of size~$(n-1)\times(n-1)$ and all non-indicated entries are zeros.
It follows that~$tA-A^\top$ is of the form
\[
tA-A^\top =\ \begin{blockarray}{crrrrrrrc}
    \begin{block}{(crrr|rrrrc)}
      \ & & \ & \  & -1 &  & & \\
      \ & & \ & \ & \ast  & t & & \\ 
       & \ &  & \ & \ast & \ast & -1 & \\
       & \ & & \ & \ast & \ast & \ast & \ddots \\  \BAhhline{-------------~}
      t & \ast & \ast & \ast  &  &  & & \\
       & -1 & \ast & \ast &  &  & &  \\
      & & t & \ast&   &  & \ast & \\
       & &  & \ddots &   &  & & \\
    \end{block}
  \end{blockarray}
\]
and thus has determinant~$\pm t^{n-1}$. This can be seen inductively by developing the first 
row and column. In particular,~$\det(tA-A^\top)$ is a unit in~$\mathbf{Z}[t^{\pm1}]$.
This finishes the proof.
\end{proof}

\end{document}